\documentclass{amsart}\usepackage{amsthm,amssymb,bm}

\theoremstyle{plain}
\newtheorem{Thm}{Theorem}
\newtheorem{Prop}[Thm]{Proposition}
\newtheorem{Lem}[Thm]{Lemma}
\newtheorem{ex}[Thm]{Example}

\newcommand{\z}{\textstyle}
\newcommand{\ve}{\varepsilon}\newcommand{\vk}{\varkappa}
\newcommand{\ia}{{\rm a}}\newcommand{\id}{{\rm d}}\newcommand{\tr}{{\rm tr}}
\newcommand{\iC}{{\rm C}}\newcommand{\iF}{{\rm F}}\newcommand{\iT}{{\rm T}}
\newcommand{\cM}{\mathcal{M}}\newcommand{\cP}{\mathcal{P}}\newcommand{\cS}{\mathcal{S}}
\newcommand{\cV}{\mathcal{V}}
\newcommand{\C}{\mathbb{C}}\newcommand{\N}{\mathbb{N}}\renewcommand{\P}{\mathbb{P}}
\newcommand{\Q}{\mathbb{Q}}\newcommand{\T}{\mathbb{T}}\newcommand{\Z}{\mathbb{Z}}
\newcommand{\BM}{\begin{smallmatrix}}\newcommand{\EM}{\end{smallmatrix}}
\newcommand{\BC}{\begin{cases}}\newcommand{\EC}{\end{cases}}
\newcommand{\VS}[1]{\z\sum\limits_{#1}}
\newcommand{\VP}[1]{\z\prod\limits_{#1}}
\newcommand{\VT}[1]{\z\bigoplus\limits_{#1}}
\newcommand{\ang}[1]{\langle #1\rangle}\newcommand{\tu}{\bigtriangleup}
\newcommand{\6}{\;\;\;\;\;\;}

\begin{document}

\title{Trace formulas of the Hecke operator on the spaces of newforms}
\author{SUDA Tomohiko}
\address{Department of Mathematics, Hokkaido University,
Sapporo 060-0810, Japan}
\email{s053019@math.sci.hokudai.ac.jp}

\maketitle

\section{Introduction}

For a positive integer $N$ and an even positive integer $k$, let $\cS_k(N)$ be the space of all cuspforms of weight $k$ with respect to the congruence subgroup
\[\Gamma_0(N)=\{(\BM a&b\\c&d\EM)\in{\rm SL}_2(\Z) \,|\, c\in N\Z\}\]
of level $N$, and $\cS_k^0(N)$ the subspace of $\cS_k(N)$ consisting of all newforms (cf. \cite[Definition 5.6.1]{DS}). For a positive integer $l$, we denote by $\iT(l)$ the $l$-th Hecke operator on $\cS_k(N)$. The first purpose of this paper is to write down on $\tr(\iT(l)|_{\cS_k^0(N)})$, for square-free $l$ and general $N$:
\begin{Thm}\label{1}If $l>1$ is square-free and $(l,N/\!\!/l)=1$, then we have
\[\tr(\iT(l)|_{\cS_k^0(N)})=-\VS{t\in\T(l)}a_{t,l,k}h_{t,l}\Lambda_{t,l}(N)+\delta_{k,2}\mu(N)\VP{p\in\P(l/\!\!/N)}(1+p).\]
Here $m/\!\!/n=m/(m,n)$, $\P(n)$ is the set of all prime divisors of $n$, $\mu$ the M\"obius function,
\[\delta_{x,y}=\BC 1 & \textit{if }x=y, \\ 0 & \textit{if }x\ne y,\EC\]
and $\T(l)=\T_-(l)\cup\T_\square(l)$ with
\begin{align*}\T_-(l)&=\{t\in\Z \,|\, t^2-4l<0\},\\
\T_\square(l)&=\{t\in\Z \,|\, t^2-4l\text{ is square}\}.\end{align*}
For each $t\in\T_-(l)$ $(\text{resp. }\T_\square(l))$, we put
\[a_{t,l,k}=\frac{\zeta^{k-1}-\eta^{k-1}}{\zeta-\eta}\;\bigg(\text{resp. }\frac{\min\{\zeta^{k-1},\eta^{k-1}\}}{2|\zeta-\eta|}\bigg),\]
with $\zeta,\eta$ two roots of $X^2-tX+l=0$,
\[h_{t,l}=\frac{h(t^2-4l)}{\text{the number of units in }\Q(\sqrt{t^2-4l})}\;(\text{resp. }1).\]
For the definition of $\Lambda_{t,l}$, see \S4.
\end{Thm}

Theorem \ref{1} has been conjectured by Kazuhide Kubo \cite{K}, for $N$ square-free, prime to 6, and $l|N$. We note that if $(l,N/\!\!/l)\ne1$ then $\tr(\iT(l)|_{\cS^0_k(N)})=0$ (cf. \cite[Theorem 4.6.17(3)]{Mi}). Theorem \ref{1} is derived from the Atkin-Lehner theory (cf \cite{AL} or \cite[\S5]{DS}) and the Eichler-Selberg trace formula (\cite[Theorem 6.8.4]{Mi} or \cite{E}) on $\iT(l)|_{\cS_k(N)}$. Some explicit calculations on $\Lambda_{t,l}$ is performed in \S5 and we write down concrete examples of Theorem \ref{1} for some $l$ and $N\le42$ in \S6.

For $f\in\cS^0_k(N)$, we denote $\ia_n(f)$ the $n$-th Fourier coefficient of $f$. We say $f$ is primitive if $f|\iT(l)=\ia_l(f)f$ for all positive integers $l$, and put $\cP_k(N)$ the set of all primitive forms in $\cS_k^0(N)$. We put $\vk=\frac k2-1$ throughout this paper. With $N^\times=\VP{p\in\P(N),\,p^2\nmid N}p$, we define for each positive divisor $i$ of $N^\times$,
\[\cP_k(N;i)=\big\{f\in\cP_k(N) \,\big|\, p\in\P(i)\Longrightarrow\ia_p(f)=-p^\vk,\;p\in\P(\z\frac{N^\times}i)\Longrightarrow\ia_p(f)=p^\vk\big\},\]
\[\cS_k^0(N;i)=\VT{f\in\cP_k(N;i)}\C f.\]
Then we see $\cP_k(N)=\sqcup_{i|N^\times}\cP_k(N;i)$ (cf. Miyake[6, Theorem 4.6.17]), thus
\[\cS^0_k(N)=\VT{f\in\cP_k(N)}\C f=\VT{i|N^\times}\cS_k^0(N;i),\]
and $\cS_k^0(N;i)$ is a Hecke submodule of $\cS_k^0(N)$ since so is $\C f$. We obtain
\begin{Thm}\label{2}
If $i|N^\times$ and $(l,N^\times)=1$, then we have
\[\tr(\iT(l)|_{\cS_k^0(N;i)})=\z\frac1{\sigma(N^\times)}\VS{h|N^\times}\ang{h,i}h^{-\vk}\tr(\iT(hl)|_{\cS^0_k(N)})\]
where $\sigma(n)$ is the number of all divisors of $n$ and $\ang{h,i}=(-1)^{\#(\P(h)\cap\P(i))}$.
\end{Thm}

Theorem \ref{2} has been conjectured also by Kubo \cite{K}, for $N$ square-free, prime to 6 and $l=1$. We give a proof of Theorem \ref{2} in \S7. Remark that
\[f\in\cP_k(N;i),\,\nu\in{\rm Gal}(\overline{\Q}/\Q)\Longrightarrow f^\nu\in\cP_k(N;i)\]
where $f^\nu=\VS{n\in\N}\nu(\ia_n(f))q^n\;\big(q=e^{2\pi\sqrt{-1}z}\big)$, since $f^\nu\in\cP_k(N)$ and $\nu(\pm p^\vk)=\pm p^\vk$. In section 8, as an application of Theorem \ref{1} and \ref{2}, we calculate $\dim(\cS_k^0(N;i))$ for $N\le42$. By virtue of our trace formulas, the Fourier coefficients of each primitive form may be calculated, in particular, we decide some primitive forms in terms of some Eisenstein series for $N=14$, in the last section.

\section{Dimension of $\cS^0_k(N)$}

For reader's convenience, we review basic facts on arithmetic functions and show Martin's formula (cf \cite[Theorem 1]{Ma}).

For each $f,g:\N\to\Q$ we define the convolution product $f*g:\N\to\Q$ by
\[(f*g)(x)=\VS{d|x}f(\frac xd)g(d).\]
Then we see the set of all functions $\N\to\Q$ is a ring under this product with $\delta=\delta_{\bullet,1}$ unit element. We put $\mathit1(x)=1$, then we see $\mathit1*\mu=\delta$. We define
\[f[l](x)=\BC f(x)& \textit{if }(x,l)=1, \\ 0& \textit{if }(x,l)\ne1, \EC\]
then we see $f[l]*g[l]=(f*g)[l]$. We say that $f:\N\to\Q$ is multiplicative if $f(1)=1$ and
\[(m,n)=1 \Longrightarrow f(mn)=f(m)f(n).\]
We see that $\delta,\mathit1,\mu$ are multiplicative, and if $f,g$ is multiplicative then so are $f*g$, $f[l]$. In particular, $\mu*\mu[l]$ is also multiplicative. 
\begin{Lem}\label{conv}
For each $n\ge1$ and prime number $p$, we see
\[(\mu*\mu)(p^n)=\BC -2 &\text{if }n=1, \\ 1 &\text{if }n=2, \\ 0&\text{if }n\ge3,\EC\]
in addition if $p|l$ then
\[(\mu*\mu[l])(p^n)=-\delta_{n,1}.\]
\end{Lem}

For each integer $d$, we denote by $K_d$ the multiplicative function such that
\[K_d(p^n)=\BC (\frac dp)-1&\text{if }n=1,\\
-(\frac dp)&\text{if }n=2,\,p\nmid d,\\
-1&\text{if }n=2,\,p|d,\\
1&\text{if }n=3,\,p|d,\\
0&\text{otherwise}\EC\]
for each prime number $p$, where $(\frac{}p)$ is the Kronecker symbol.

For each $f\in\cS_k(N)$ and $h\in\N$, we define $f^{(h)}(z)=f(hz)$. Then, we see
\[\cS_k(N)=\VT{M|N}\VT{h|\frac NM}\cS_k(M)^{(h)}\]
by the Atkin-Lehner theory, and thus
\[\dim\cS_k(x)=\VS{M|x}\VS{h|\frac NM}\dim\cS_k^0(M)=(\textit1*\textit1*\dim\cS_k^0)(x),\]
i.e., $\dim\cS^0_k=\mu*\mu*\dim\cS_k$. On the other hand, we have
\[\dim\cS_k=\z\frac{k-1}{12}\ve_\id+\frac14(-1)^{\frac k2}\ve_2-\frac 13(\frac{k-1}3)\ve_3-\frac12\ve_\infty+\delta_{k,2}\textit1,\]
where $\ve_\id,\ve_2,\ve_3$ are the multiplicative functions satisfying
\[\ve_\id(p^n)=p^n+p^{n-1},\]
\[\ve_2(2^n)=\delta_{n,1},\6\ve_2(p^n)=1+(\z\frac{-1}p)\text{ if }p\ne2,\]
\[\ve_3(3^n)=\delta_{n,1},\6\ve_3(p^n)=1+(\z\frac{-3}p)\text{ if }p\ne3,\]
\[\ve_\infty(p^n)=p^{[\frac n2]}+p^{[\frac{n-1}2]}\]
for each $n\ge1$ and prime $p$ (cf. \cite[Theorem 3.5.1]{DS}). We note $\ve_2(p^n)=1+(\z\frac{-4}p)$ if $p\ne2$. It follows from Lemma \ref{conv} that
\[\dim\cS^0_k=\z\frac{k-1}{12}\mu*\mu*\ve_\id+\frac14(-1)^{\frac k2}K_{-4}-\frac 13(\frac{k-1}3)K_{-3}-\frac12\mu*\mu*\ve_\infty+\delta_{k,2}\mu,\]
\[(\mu*\mu*\ve_\id)(p^n)=\BC p-1&\text{if }n=1,\\
p^2-p-1&\text{if }n=2,\\
p^{n-3}(p-1)(p^2-1)&\text{if }n\ge3,\EC\]
\[(\mu*\mu*\ve_\infty)(p^n)=\BC0&\text{if }n\text{ is odd},\\
p-2&\text{if }n=2,\\
p^{\frac n2-2}(p-1)^2&\text{if }n\text{ is even},\ge4.\EC\]
\begin{ex}
For each $N\le42$, we write down the following formulas:
\begin{align*}\dim\cS^0_k(1)&=\z\frac1{12}(k-1)+\frac14(-1)^{\frac k2}-\frac13(\frac{k-1}3)-\frac12+\delta_{k,2},\\
\dim\cS^0_k(2)&=\z\frac1{12}(k-1)-\frac14(-1)^{\frac k2}+\frac23(\frac{k-1}3)-\delta_{k,2},\\
\dim\cS^0_k(3)&=\z\frac16(k-1)-\frac12(-1)^{\frac k2}+\frac13(\frac{k-1}3)-\delta_{k,2},\\
\dim\cS^0_k(4)&=\z\frac1{12}(k-1)-\frac14(-1)^{\frac k2}-\frac13(\frac{k-1}3),\\
\dim\cS^0_k(5)&=\z\frac13(k-1)+\frac23(\frac{k-1}3)-\delta_{k,2},\\
\dim\cS^0_k(6)&=\z\frac16(k-1)+\frac12(-1)^{\frac k2}-\frac23(\frac{k-1}3)+\delta_{k,2},\\
\dim\cS^0_k(7)&=\z\frac12(k-1)-\frac12(-1)^{\frac k2}-\delta_{k,2},\\
\dim\cS^0_k(8)&=\z\frac14(k-1)+\frac14(-1)^{\frac k2},\\
\dim\cS^0_k(9)&=\z\frac5{12}(k-1)+\frac14(-1)^{\frac k2}+\frac13(\frac{k-1}3)-\frac12,\\
\dim\cS^0_k(10)&=\z\frac13(k-1)-\frac43(\frac{k-1}3)+\delta_{k,2},\\
\dim\cS^0_k(11)&=\z\frac56(k-1)-\frac12(-1)^{\frac k2}+\frac23(\frac{k-1}3)-\delta_{k,2},\\
\dim\cS^0_k(12)&=\z\frac16(k-1)+\frac12(-1)^{\frac k2}+\frac13(\frac{k-1}3),\\
\dim\cS^0_k(13)&=(k-1)-\delta_{k,2},\\
\dim\cS^0_k(14)&=\z\frac12(k-1)+\frac12(-1)^{\frac k2}+\delta_{k,2},\\
\dim\cS^0_k(15)&=\z\frac23(k-1)-\frac23(\frac{k-1}3)+\delta_{k,2},\\
\dim\cS^0_k(16)&=\z\frac k2-1,\\
\dim\cS^0_k(17)&=\z\frac43(k-1)+\frac23(\frac{k-1}3)-\delta_{k,2},\\
\dim\cS^0_k(18)&=\z\frac5{12}(k-1)-\frac14(-1)^{\frac k2}-\frac23(\frac{k-1}3),\\
\dim\cS^0_k(19)&=\z\frac32(k-1)-\frac12(-1)^{\frac k2}-\delta_{k,2},\\
\dim\cS^0_k(20)&=\z\frac13(k-1)+\frac23(\frac{k-1}3),\\
\dim\cS^0_k(21)&=(k-1)+(-1)^{\frac k2}+\delta_{k,2},\\
\dim\cS^0_k(22)&=\z\frac56(k-1)+\frac12(-1)^{\frac k2}-\frac43(\frac{k-1}3)+\delta_{k,2},\\
\dim\cS^0_k(23)&=\z\frac{11}6(k-1)-\frac12(-1)^{\frac k2}-\frac23(\frac{k-1}3)-\delta_{k,2},\\
\dim\cS^0_k(24)&=\z\frac12(k-1)-\frac12(-1)^{\frac k2},\\
\dim\cS^0_k(25)&=\z\frac{19}{12}(k-1)-\frac14(-1)^{\frac k2}-\frac13(\frac{k-1}3)-\frac32,\\
\dim\cS^0_k(26)&=(k-1)+\delta_{k,2},\\
\dim\cS^0_k(27)&=\z\frac43(k-1)-\frac13(\frac{k-1}3),\\
\dim\cS^0_k(28)&=\z\frac12(k-1)+\frac12(-1)^{\frac k2},\\
\dim\cS^0_k(29)&=\z\frac73(k-1)+\frac23(\frac{k-1}3)-\delta_{k,2},\\
\dim\cS^0_k(30)&=\z\frac23(k-1)+\frac43(\frac{k-1}3)-\delta_{k,2},\\
\dim\cS^0_k(31)&=\z\frac52(k-1)-\frac12(-1)^{\frac k2}-\delta_{k,2},\\
\dim\cS^0_k(32)&=k-1,\end{align*}\begin{align*}
\dim\cS^0_k(33)&=\z\frac53(k-1)+(-1)^{\frac k2}-\frac23(\frac{k-1}3)+\delta_{k,2},\\
\dim\cS^0_k(34)&=\z\frac43(k-1)-\frac43(\frac{k-1}3)+\delta_{k,2},\\
\dim\cS^0_k(35)&=2(k-1)+\delta_{k,2},\\
\dim\cS^0_k(36)&=\z\frac5{12}(k-1)-\frac14(-1)^{\frac k2}+\frac13(\frac{k-1}3),\\
\dim\cS^0_k(37)&=3(k-1)-\delta_{k,2},\\
\dim\cS^0_k(38)&=\z\frac32(k-1)+\frac12(-1)^{\frac k2}+\delta_{k,2},\\
\dim\cS^0_k(39)&=2(k-1)+\delta_{k,2},\\
\dim\cS^0_k(40)&=k-1,\\
\dim\cS^0_k(41)&=\z\frac{10}3(k-1)+\frac23(\frac{k-1}3)-\delta_{k,2},\\
\dim\cS^0_k(42)&=(k-1)-(-1)^{\frac k2}-\delta_{k,2}.\end{align*}\end{ex}

\section{Relation lemma}

We define the $l$-th Hecke operator $\iT_N(l)$ on $\cS_k(N)$ by
\[f|\iT_N(l)=l^{k-1}\VS{d|l,\,(l/d,N)=1}\VS{b=0}^{d-1}d^{-k}f\big((\frac ldz+b)/d\big)\]
(cf. \cite[4.5.26]{Mi}).
First, the following relations of Hecke operators on different levels are easily shown(cf. \cite[in the proof of Proposition 5.6.2]{DS}):
\begin{Lem}\label{HO}
Suppose that $f\in\cS_k(N)$.\\[1mm]
\6(a) If $p,q$ are prime numbers with $p\ne q$, then we see
\[f|\iT_{pN}(q)=f|\iT_N(q),\]
\[f^{(p)}|\iT_{pN}(q)=(f|\iT_N(q))^{(p)}.\]\\[-2mm]
\6(b) If $p$ is a prime number with $(p,N)=1$, then we see
\[f^{(p)}|\iT_{pN}(p)=f.\]
\end{Lem}

For a divisor $M$ of $N$ and $f\in\cP_k(M)$, we put
\[\cV(f,N)=\VT{h|\frac NM}\C\cdot f^{(h)}.\]
\begin{Lem}Suppose $M|N$ and $f=\VS{n=1}^\infty a_nq^n\in\cP_k(M)$. If $l$ is square-free, then we have
\[\tr(\iT_N(l)|_{\cV(f,N)})=(\textit1*\textit1[l])(\z\frac NM)a_l.\]
\end{Lem}
\begin{proof}For $h|\frac NM$, we show
\[f^{(h)}|\iT_N(l)=a_{l/\!\!/h}f^{(h/\!\!/l)}.\]
by induction on $\#\P(l)$. First, if $\#\P(l)=0$, that is $l=1$, then we get the assertion since $\iT_N(1)$ is the identity operator and $a_1=1$. Next, suppose the assertion is true for $l$ and take a prime $p\nmid l$, then we see
\[f^{(h)}|\iT_N(p)=\BC a_pf^{(h)}&\text{if }p\nmid h,\\ f^{(h/p)}&\text{if }p|h\EC\]
by Lemma \ref{HO}, thus
\[f^{(h)}|\iT_N(lp)=a_{l/\!\!/h}f^{(h/\!\!/l)}|\iT_N(p)=a_{lp/\!\!/h}f^{(h/\!\!/lp)}.\]

Now, taking a basis $\{f^{(h)}\}_{h|\frac NM}$ of $\cV(f,N)$, we see
\[\tr(\iT_N(l)|_{\cV(f,N)})=\VS{h|\frac NM}\delta_{h,h/\!\!/l}a_{l/\!\!/h}=\VS{h|\frac NM}\textit1[l](h)a_l,\]
thus we get the assertion.\end{proof}

\begin{Lem}\label{THO}
If $l$ is square-free, then we have
\[\tr^0_{k,l}=\mu*\mu[l]*\tr_{k,l},\]
where we define the function $\tr_{k,l}$ by $x\mapsto\tr(\iT(l)|_{\cS_k(x)})$ and $\tr_{k,l}^0$ by $x\mapsto\tr(\iT(l)|_{\cS_k^0(x)})$.
\end{Lem}
\begin{proof}
By the above Lemma and the Atkin-Lehner theory, we have
\begin{align*}\tr_{k,l}(x)&=\VS{M|x}\VS{f\in\cP_k(M)}\tr(\iT_N(l)|_{\cV(f,N)})\\
&=\VS{M|x}(\textit1*\textit1[l])(\frac NM)\tr(\iT_M(l)|_{\cS_k^0(M)})\\
&=(\textit1*\textit1[l]*\tr^0_{k,l})(x).\end{align*}
\end{proof}

\section{Proof of Theorem \ref{1}}

For each $t\in\T(l)$, we put $d(t,l)$ the discriminant of $\Q\big(\sqrt{t^2-4l}\big)$ and $m(t,l)=\sqrt\frac{t^2-4l}{d(t,l)}$.
With $d=d(t,l)$ and $m=m(t,l)$, for each $\phi|m$, we define
\[b_{t,l,\phi}=\phi\VP{p\in\P(\phi)}\big(1-(\frac dp)p^{-1}\big),\]
and the multiplicative function $c_{t,l,\phi}$ by
\[c_{t,l,\phi}(p^n)=\VS{\xi\in\Omega/\psi p^n}\bm1_p(\xi)+\VS{\xi\in\Omega'/\psi p^n}\bm1_p(t-\xi)\]
for each $n\ge1$ and prime $p$, where $\psi=\frac m\phi$, $\bm1_p$ is the trivial character mod $p$ and
\begin{align*}\Omega&=\{\xi\in\Z_p \,|\, \xi^2-t\xi+l\equiv0 \bmod \psi^2p^n\Z_p\},\\
\Omega'&=\{\xi\in\Z_p \,|\, \xi^2-t\xi+l\equiv0 \bmod \psi^2p^{n+1}\Z_p\}\text{ if }p|d\phi,\;=\emptyset\text{ otherwise}.\end{align*}
We note that if $t\in\T_\square(l)$ then $d=1$ and $b_{t,l,\phi}=\#(\Z/\phi\Z)^\times$.

Suppose that $l>1$ is square-free and define the multiplicative function $\omega_l$ by
\[\omega_l(p^n)=\BC\frac p{1+p}&\text{if }p\in\P(l),\\1&\text{if }p\notin\P(l).\EC\]
The Eichler-Selberg trace formula says
\[\tr_{k,l}=-\VS{t\in\T(l)}a_{t,l,k}h_{t,l}\VS{\phi|m(t,l)}b_{t,l,\phi}c_{t,l,\phi}+\delta_{k,2}\VP{p\in\P(l)}(1+p)\cdot\omega_l.\]
Now, we define
\[\Lambda_{t,l}=\VS{\phi|m(t,l)}b_{t,l,\phi}(\mu*\mu[l]*c_{t,l,\phi}),\]
then we have
\[\tr_{k,l}^0=-\VS{t\in\T(l)}a_{t,l,k}h_{t,l}\Lambda_{t,l}+\delta_{k,2}\VP{p\in\P(l)}(1+p)\cdot(\mu*\mu[l]*\omega_l).\]
and
\[(\mu*\mu[l]*\omega_l)(p)=-\z\frac1{1+p}\6\text{if }p\in\P(l),\]
\[(\mu*\mu[l]*\omega_l)(p^n)=-\delta_{n,1}\6\text{if }p\notin\P(l).\]
If $(l,N/\!\!/l)=1$, then we obtain
\[\tr_{k,l}^0(N)=-\VS{t\in\T(l)}a_{t,l,k}h_{t,l}\Lambda_{t,l}(N)+\delta_{k,2}\mu(N)\VP{p\in\P(l/\!\!/N)}(1+p).\]
 
\section{Calculations on $\Lambda$}

\begin{Prop}\label{Lambda0}
Suppose that $l>1$ is square-free, $(l,N/\!\!/l)=1$ and $(l,N)\nmid t$. Then we have
\[\Lambda_{t,l}(N)=0.\]
\end{Prop}
\begin{proof}
We see that there exists a prime number $p$ such that $p|(l,N)$ and $p\nmid t$. We note $p^2\nmid N$ since $(l,N/\!\!/l)=1$. We show for each $\phi|m(t,l)$, $(\mu*\mu[l]*c_{t,l,\phi})(p)=0$. Indeed, we see $p\nmid (t^2-4l)$ and
\[\Omega=\{\xi\in\Z_p \,|\, \xi^2-t\xi\equiv0 \bmod p\Z_p\}=\{\xi\in\Z_p \,|\, \xi\equiv0,t \bmod p\Z_p\},\]
$\Omega'=\emptyset$, hence $c_{t,l,\phi}(p)=\bm1_p(0)+\bm1_p(t)=1$. We get the assertion by Lemma \ref{conv}.
\end{proof}
\begin{Prop}\label{Lambda}
Suppose that $l>1$ is square-free and $(l,N)|t$. Put $d=d(t,l)$ and $m=m(t,l)$. If $m=1$, then we have
\[\Lambda_{t,l}(N)=K_d(N).\]
If $m$ is prime, then putting $v_m(N)=\max\{n \,|\, N\in m^n\Z\}$ we have
\[\Lambda_{t,l}(N)=K_d(Nm^{-v_m(N)})\times\BC m+1-(\frac dm)&\text{if }v_m(N)=0,\\
(\frac dm)-1&\text{if }v_m(N)=1,\\
m^2-2m-1+(\frac dm)&\text{if }v_m(N)=2,\\
(m-(\frac dm))(m-1)((\frac dm)-1)&\text{if }v_m(N)=3,\,m\nmid d,\\
-(m-(\frac dm))m(\frac dm)&\text{if }v_m(N)=4,\,m\nmid d,\\
1-m^2&\text{if }v_m(N)=3,\,m|d,\\
m(1-m)&\text{if }v_m(N)=4,\,m|d,\\
m^2&\text{if }v_m(N)=5,\,m|d,\\
0&\text{otherwise}.\EC\]
\end{Prop}
\begin{proof}For each $p|(l,N)$ and $\phi|m$, we see $c_{t,l,\phi}(p)=0$ since
\[\Omega\subset\{\xi\in\Z_p \,|\, \xi^2\equiv0 \bmod p\Z_p\}=p\Z_p,\]
\[\Omega'\subset\{\xi\in\Z_p \,|\, \xi^2-t\xi+l\equiv0 \bmod p^2\Z_p\}=\emptyset,\]
and thus $(\mu*\mu[l]*c_{t,l,\phi})(p)=-1$ by Lemma \ref{conv}. Therefore, the first assertion follows from Lemma \ref{p} below and
\[\Lambda_{t,l}=\mu*\mu[l]*c_{t,l,1}.\]
The second one also follows from Lemmas \ref{p}, \ref{p2} below and
\[\Lambda_{t,l}=\mu*\mu[l]*c_{t,l,1}+\big(m-(\z\frac dm)\big)\mu*\mu[l]*c_{t,l,m}.\]
\end{proof}

In Lemmas \ref{p} and \ref{p2}, assume that $l>1$ is square-free, $t\in\T(l)$, $n\ge1$, and $p$ is a prime number, and put $d=d(t,l)$ and $m=m(t,l)$. We note $d\equiv0,1\bmod4$ and $t^2-4l=dm^2$.

\begin{Lem}\label{p}
If $p\nmid lm$, then for each $\phi|m$, we have
\[(\mu*\mu[l]*c_{t,l,\phi})(p^n)=K_d(p^n).\]
\end{Lem}
\begin{proof}We show
\[c_{t,l,\phi}(p^n)=\BC 1+(\frac dp)&\text{if }p\nmid d, \\ \delta_{n,1}&\text{if }p|d,\EC\]
then we get the assertion by Lemma \ref{conv}. First, suppose $p\nmid d$. If $p\ne2$, then we see
\[c_{t,l,\phi}(p^n)=\#(\{\xi\in\Z_p \,|\, (\xi-\z\frac t2)^2\equiv\frac{dm^2}4 \bmod p^n\}/p^n)=1+(\frac dp).\]
If $p=2$, then we see $d\equiv dm^2\equiv t^2-4\equiv5\bmod8$ and
\[c_{t,l,1}(2)=\#\{\xi\in\Z_2 \,|\, \xi^2-\xi+1\equiv0 \bmod 2\}/2)=0=1+(\z\frac d2).\]
We easily see $c_{t,l,1}(2^n)=0$ for $n\ge2$. Next, suppose $p|d$. If $p\ne2$, then we see
\begin{align*}c_{t,l,1}(p)&=\#(\{\xi\in\Z_p \,|\, (\xi-\z\frac t2)^2\equiv0 \bmod p\}/p)\\[-2pt]
&\6+\#(\{\xi\in\Z_p \,|\, (\xi-\z\frac t2)^2\equiv\frac{dm^2}4 \bmod p^2\}/p)\\
&=1+0,\end{align*}
and $c_{t,l,1}(p^n)=0$ for $n\ge2$. If $p=2$, then we see $2|t$, $\frac d4\equiv2,3\bmod4$ and
\begin{align*}c_{t,l,1}(2)&=\#(\{\xi\in\Z_2 \,|\, \xi^2+1\equiv0 \bmod 2\}/2)\\[-2pt]
&\6+\#(\{\xi\in\Z_2 \,|\, (\xi-\z\frac t2)^2\equiv\frac{dm^2}4 \bmod 4\}/2)\\
&=1+0,\end{align*}
and $c_{t,l,1}(2^n)=0$ for $n\ge2$.
\end{proof}

\begin{Lem}\label{p2}
If $m$ is prime and $m\nmid l$, then we have
\[(\mu*\mu[l]*c_{t,l,1})(m^n)=K_d(m^n),\]
and
\[(\mu*\mu[l]*c_{t,l,m})(m^n)=\BC0&\text{if }n=1,\\
m-2+(\frac dm)&\text{if }n=2,\\
(m-1)((\frac dm)-1)&\text{if }n=3,\,m\nmid d,\\
-m(\frac dm)&\text{if }n=4,\,m\nmid d,\\
-m&\text{if }n=3,\,m|d,\\
1-m&\text{if }n=4,\,m|d,\\
m&\text{if }n=5,\,m|d,\\
0&\text{otherwise}.\EC\]
\end{Lem}
\begin{proof}
The case $m=2$: We note $2|t$. If $2\nmid d$, then we see $4|t$,
\[c_{t,l,1}(2^n)=\#(\{\xi\in\Z_2 \,|\, (\xi-\z\frac t2)^2\equiv d \bmod 2^{n+2}\}/2^{n+1})=1+(\z\frac d2),\]
and
\begin{align*}c_{t,l,2}(2^n)&=\#(\{\xi\in\Z_2 \,|\, (\xi-\z\frac t2)^2\equiv d \bmod 2^n\}/2^n)\\[-2pt]
&\6+\#(\{\xi\in\Z_2 \,|\, (\xi-\z\frac t2)^2\equiv d \bmod 2^{n+1}\}/2^n)\\
&=\BC 2&\text{if }n=1,\\3+(\frac d2)&\text{if }n=2,\\3(1+(\frac d2))&\text{if }n\ge3.\EC\end{align*}
If $2|d$, then we see $4\nmid t,\,\frac d4\equiv2,3\bmod4$,
\begin{align*}c_{t,l,1}(2^n)&=\#(\{\xi\in\Z_2 \,|\, (\xi-\z\frac t2)^2\equiv d\bmod 2^{n+2}\}/2^{n+1})\\[-2pt]
&\6+\#(\{\xi\in\Z_2 \,|\, (\xi-\z\frac t2)^2\equiv d\bmod 2^{n+3}\}/2^{n+1})\\
&=\delta_{n,1},\end{align*}
and
\begin{align*}c_{t,l,2}(2^n)&=\#(\{\xi\in\Z_2 \,|\, (\xi-\z\frac t2)^2\equiv d \bmod 2^n\}/2^n)\\[-2pt]
&\6+\#(\{\xi\in\Z_2 \,|\, (\xi-\z\frac t2)^2\equiv d \bmod 2^{n+1}\}/2^n)\\
&=\BC 2&\text{if }n=1,3,\\3&\text{if }n=2,\\
0&\text{if }n\ge4.\EC\end{align*}
The case $m\neq2$: If $m\nmid d$, then we see
\[c_{t,l,1}(m^n)=\#(\{\xi\in\Z_m \,|\, (\xi-\z\frac t2)^2\equiv\frac d4 \bmod m^n\}/m^n)=1+(\z\frac d2),\]
\begin{align*}c_{t,l,m}(m^n)&=\#(\{\xi\in\Z_m \,|\, (\xi-\z\frac t2)^2\equiv\frac{dm^2}4\bmod m^n\}/m^n)\\[-2pt]
&\6+\#(\{\xi\in\Z_m \,|\, (\xi-\z\frac t2)^2\equiv\frac{dm^2}4\bmod m^{n+1}\}/m^n)\\
&=\BC2&\text{if }n=1,\\
m+1+(\frac dm)&\text{if }n=2,\\
(m+1)(1+(\frac dm))&\text{if }n\ge3.\EC\end{align*}
If $m|d$, then we see
\begin{align*}c_{t,l,1}(m^n)&=\#(\{\xi\in\Z_m \,|\, (\xi-\z\frac t2)^2\equiv\frac d4\bmod m^n\}/m^n)\\[-2pt]
&\6+\#(\{\xi\in\Z_m \,|\, (\xi-\z\frac t2)^2\equiv\frac d4\bmod m^{n+1}\}/m^n)\\
&=\delta_{n,1},\end{align*}
\begin{align*}c_{t,l,m}(m^n)&=\#(\{\xi\in\Z_m \,|\, (\xi-\z\frac t2)^2\equiv\frac{dm^2}4\bmod m^n\}/m^n)\\[-2pt]
&\6+\#(\{\xi\in\Z_m \,|\, (\xi-\z\frac t2)^2\equiv\frac{dm^2}4\bmod m^{n+1}\}/m^n)\\
&=\BC2&\text{if }n=1,\\
m+1&\text{if }n=2,\\
m&\text{if }n=3,\\
0&\text{if }n\ge4.\EC\end{align*}
We get all the assertions by Lemma \ref{conv}.
\end{proof}

\section{Examples of trace on $\cS_k^0(N)$}

In this section, we give some concrete formulas. We first note $a_{-t,k,l}=a_{t,k,l}$, $h_{-t,l}=h_{t,l}$, $\Lambda_{-t,l}=\Lambda_{t,l}$, and
\[a_{0,l,k}=\frac{(-l)^{\frac{k-1}2}-(-l)^{\frac{k-1}2}}{(-l)^{\frac12}-(-l)^{\frac12}}=\frac{(-l)^\vk-(-l)^\vk}{1-(-1)}=(-l)^\vk.\]

\begin{Lem}\label{as}
If $l$ is prime, then we have $\T_\square(l)=\{\pm(l+1)\}$, $m(l+1,l)=l-1$ and $a_{l+1,l,k}=\z\frac1{2(l-1)}$.
\end{Lem}
\begin{proof}
If $t^2-4l=m^2$, then we see $(t+m)(t-m)=4l$, thus $t=l+1$, $m=l-1$ and $X^2-tX+l=(X-1)(X-l)$.
\end{proof}

\begin{Prop}
If $(2,N/\!\!/2)=1$ i.e. $4\nmid N$, then we get
\begin{align*}\tr(\iT(2)|_{\cS_k^0(N)})&=-\z\frac12(-2)^\vk\big(K_{-2}(N)+A_kK_{-4}(N)\big)-a_{1,2,k}K_{-7}(N)\\
&\6-K_1(N)+\delta_{k,2}\mu(N)\VP{p\in\P(2/\!\!/N)}(1+p),\end{align*}
where
\[A_k=\BC 1 & \textit{if }k\equiv0,2 \bmod 8, \\ -1 & \textit{if }k\equiv4,6\bmod8.\EC\]
\end{Prop}
\begin{proof}
We see $\T_-(2)=\{0,\pm1,\pm2\}$, $h_{0,2}=h_{1,2}=\z\frac12$, $h_{2,2}=\frac14$, and
\begin{align*}a_{2,2,k}&=\z\frac{\big(1+\sqrt{-1}\big)^{k-1}-\big(1-\sqrt{-1}\big)^{k-1}}{2\sqrt{-1}}\\
&=\z\frac{\big(\sqrt{-2}\big)^{k-1}}{2\sqrt{-1}}\Big(\big(\frac{1+\sqrt{-1}}{\sqrt2}\big)^{k-1}-\big(\frac{-1+\sqrt{-1}}{\sqrt2}\big)^{k-1}\Big)\\
&=(-2)^\vk A_k.\end{align*}
\end{proof}
We note $a_{1,2,k}=\frac{\big(\frac{1+\sqrt{-7}}2\big)^{k-1}-\big(\frac{1-\sqrt{-7}}2\big)^{k-1}}{\sqrt{-7}}$.
\begin{ex} We write down the following formulas:
\begin{align*}\tr(\iT(2)|_{\cS^0_k(1)})&=-\z\frac12(-2)^\vk(1+A_k)-a_{1,2,k}-1+3\delta_{k,2},\\
\tr(\iT(2)|_{\cS^0_k(2)})&=\z\frac12(-2)^\vk(1+A_k)-\delta_{k,2},\\
\tr(\iT(2)|_{\cS^0_k(3)})&=(-2)^\vk A_k+2a_{1,2,k}-3\delta_{k,2},\\
\tr(\iT(2)|_{\cS^0_k(5)})&=(-2)^\vk+2a_{1,2,k}-3\delta_{k,2},\\
\tr(\iT(2)|_{\cS^0_k(6)})&=-(-2)^\vk A_k+\delta_{k,2},\\
\tr(\iT(2)|_{\cS^0_k(7)})&=(-2)^\vk(1+A_k)+a_{1,2,k}-3\delta_{k,2},\\
\tr(\iT(2)|_{\cS^0_k(9)})&=\z\frac12(-2)^\vk(1-A_k)-a_{1,2,k}+1,\\
\tr(\iT(2)|_{\cS^0_k(10)})&=-(-2)^\vk+\delta_{k,2},\end{align*}\begin{align*}
\tr(\iT(2)|_{\cS^0_k(11)})&=(-2)^\vk A_k-3\delta_{k,2},\\
\tr(\iT(2)|_{\cS^0_k(13)})&=(-2)^\vk+2a_{1,2,k}-3\delta_{k,2},\\
\tr(\iT(2)|_{\cS^0_k(14)})&=-(-2)^\vk(1+A_k)+\delta_{k,2},\\
\tr(\iT(2)|_{\cS^0_k(15)})&=-4a_{1,2,k}+3\delta_{k,2},\\
\tr(\iT(2)|_{\cS^0_k(17)})&=2a_{1,2,k}-3\delta_{k,2},\\
\tr(\iT(2)|_{\cS^0_k(18)})&=-\z\frac12(-2)^{\vk}(1-A_k),\\
\tr(\iT(2)|_{\cS^0_k(19)})&=(-2)^\vk A_k+2a_{1,2,k}-3\delta_{k,2},\\
\tr(\iT(2)|_{\cS^0_k(21)})&=-2(-2)^\vk A_k-2a_{1,2,k}+3\delta_{k,2},\\
\tr(\iT(2)|_{\cS^0_k(22)})&=-(-2)^\vk A_k+\delta_{k,2},\\
\tr(\iT(2)|_{\cS^0_k(23)})&=(-2)^\vk(1+A_k)-3\delta_{k,2},\\
\tr(\iT(2)|_{\cS^0_k(25)})&=-\z\frac12(-2)^\vk(1-A_k)-a_{1,2,k}+1,\\
\tr(\iT(2)|_{\cS^0_k(26)})&=-(-2)^\vk+\delta_{k,2},\\
\tr(\iT(2)|_{\cS^0_k(27)})&=0,\\
\tr(\iT(2)|_{\cS^0_k(29)})&=(-2)^\vk-3\delta_{k,2},\\
\tr(\iT(2)|_{\cS^0_k(30)})&=-\delta_{k,2},\\
\tr(\iT(2)|_{\cS^0_k(31)})&=(-2)^\vk(1+A_k)+2a_{1,2,k}-3\delta_{k,2},\\
\tr(\iT(2)|_{\cS^0_k(33)})&=-2(-2)^\vk A_k+3\delta_{k,2},\\
\tr(\iT(2)|_{\cS^0_k(34)})&=\delta_{k,2},\\
\tr(\iT(2)|_{\cS^0_k(35)})&=-2(-2)^\vk-2a_{1,2,k}+3\delta_{k,2},\\
\tr(\iT(2)|_{\cS^0_k(38)})&=-(-2)^\vk A_k+\delta_{k,2},\\
\tr(\iT(2)|_{\cS^0_k(42)})&=2(-2)^\vk A_k-\delta_{k,2}.\end{align*}
\end{ex}

\begin{Prop}
If $(3,N/\!\!/3)=1$ i.e. $9\nmid N$, then we get
\begin{align*}\tr(\iT(3)|_{\cS_k^0(N)})&=-\z\frac16(-3)^\vk\big(\Lambda_{0,3}(N)+2B_kK_{-3}(N)\big)-a_{1,3,k}K_{-11}(N)-a_{2,3,k}K_{-2}(N)\\
&\6-\z\frac12\Lambda_{4,3}(N)+\delta_{k,2}\mu(N)\VP{p\in\P(3/\!\!/N)}(1+p),\end{align*}
where
\[B_k=\BC 1 & \textit{if }k\equiv0,2 \bmod 6, \\ -2 & \textit{if }k\equiv4\bmod6.\EC\]
In addition, we see
\[\Lambda_{0,3}(N)=K_{-3}(N2^{-v_2(N)})\times\BC4&\text{if }v_2(N)=0,\\
-2&\text{if }v_2(N)=1,2,\\
-6&\text{if }v_2(N)=3,\\
6&\text{if }v_2(N)=4,\\
0&\text{if }v_2(N)\ge5,\EC\]
\[\Lambda_{4,3}(N)=K_1(N2^{-v_2(N)})\times\BC 2&\text{if }v_2(N)=0,\\
-2&\text{if }v_2(N)=4,\\
0&\text{otherwise}.\EC\]
\end{Prop}
\begin{proof}
We see $\T_-(2)=\{0,\pm1,\pm2,\pm3\}$, $h_{0,3}=h_{3,3}=\z\frac16$, $h_{1,3}=h_{2,3}=\frac12$, and
\begin{align*}a_{3,3,k}&=\z\frac{\big(\frac{3+\sqrt{-3}}2\big)^{k-1}-\big(\frac{3-\sqrt{-3}}2\big)^{k-1}}{\sqrt{-3}}\\
&=\z\frac{\big(\sqrt{-3}\big)^{k-1}}{\sqrt{-3}}\Big(\big(\frac{1-\sqrt{-3}}2\big)^{k-1}-\big(\frac{-1-\sqrt{-3}}2\big)^{k-1}\Big)\\
&=B_k(-3)^\vk.\end{align*}
\end{proof}
We note $a_{1,3,k}=\frac{\big(\frac{1+\sqrt{-11}}2\big)^{k-1}-\big(\frac{1-\sqrt{-11}}2\big)^{k-1}}{\sqrt{-11}}$ and $a_{2,3,k}=\z\frac{\big(1+\sqrt{-2}\big)^{k-1}-\big(1-\sqrt{-2}\big)^{k-1}}{2\sqrt{-2}}$.
\begin{ex} We write down the following formulas:
\begin{align*}\tr(\iT(3)|_{\cS^0_k(1)})&=-\z\frac13(-3)^\vk(2+B_k)-a_{1,3,k}-a_{2,3,k}-1+4\delta_{k,2},\\
\tr(\iT(3)|_{\cS^0_k(2)})&=\z\frac13(-3)^\vk(1+2B_k)+2a_{1,3,k}+a_{2,3,k}-4\delta_{k,2},\\
\tr(\iT(3)|_{\cS^0_k(3)})&=\z\frac13(-3)^\vk(2+B_k)-\delta_{k,2},\\
\tr(\iT(3)|_{\cS^0_k(4)})&=\z\frac13(-3)^\vk(1-B_k)-a_{1,3,k}+a_{2,3,k},\\
\tr(\iT(3)|_{\cS^0_k(5)})&=\z\frac23(-3)^\vk(2+B_k)+2a_{2,3,k}-4\delta_{k,2},\\
\tr(\iT(3)|_{\cS^0_k(6)})&=-\z\frac13(-3)^\vk(1+2B_k)+\delta_{k,2},\\
\tr(\iT(3)|_{\cS^0_k(7)})&=2a_{1,3,k}+2a_{2,3,k}-4\delta_{k,2},\\
\tr(\iT(3)|_{\cS^0_k(8)})&=(-3)^\vk-a_{2,3,k},\\
\tr(\iT(3)|_{\cS^0_k(10)})&=-\z\frac23(-3)^\vk(1+2B_k)-2a_{2,3,k}+4\delta_{k,2},\\
\tr(\iT(3)|_{\cS^0_k(11)})&=\z\frac23(-3)^\vk(2+B_k)+a_{1,3,k}+4\delta_{k,2},\\
\tr(\iT(3)|_{\cS^0_k(12)})&=-\z\frac13(-3)^\vk(1-B_k),\\
\tr(\iT(3)|_{\cS^0_k(13)})&=2a_{1,3,k}+2a_{2,3,k}-4\delta_{k,2},\\
\tr(\iT(3)|_{\cS^0_k(14)})&=-4a_{1,3,k}-2a_{2,3,k}+4\delta_{k,2},\\
\tr(\iT(3)|_{\cS^0_k(15)})&=-\z\frac23(-3)^\vk(2+B_k)+\delta_{k,2},\\
\tr(\iT(3)|_{\cS^0_k(16)})&=-(-3)^\vk+1,\\
\tr(\iT(3)|_{\cS^0_k(17)})&=\z\frac23(-3)^\vk(2+B_k)+2a_{1,3,k}-4\delta_{k,2},\\
\tr(\iT(3)|_{\cS^0_k(19)})&=2a_{1,3,k}-4\delta_{k,2},\\
\tr(\iT(3)|_{\cS^0_k(20)})&=-\z\frac23(-3)^\vk(1-B_k)-2a_{2,3,k},\\
\tr(\iT(3)|_{\cS^0_k(21)})&=\delta_{k,2},\\
\tr(\iT(3)|_{\cS^0_k(22)})&=-\z\frac23(-3)^\vk(1+2B_k)-2a_{1,3,k}+4\delta_{k,2},\\
\tr(\iT(3)|_{\cS^0_k(23)})&=\z\frac23(-3)^\vk(2+B_k)+2a_{2,3,k}-4\delta_{k,2},\\
\tr(\iT(3)|_{\cS^0_k(24)})&=-(-3)^\vk,\\
\tr(\iT(3)|_{\cS^0_k(25)})&=-\z\frac13(-3)^\vk(2+B_k)+a_{1,3,k}-a_{2,3,k}+1,\end{align*}\begin{align*}
\tr(\iT(3)|_{\cS^0_k(26)})&=-4a_{1,3,k}-2a_{2,3,k}+4\delta_{k,2},\\
\tr(\iT(3)|_{\cS^0_k(28)})&=2a_{1,3,k}-2a_{2,3,k},\\
\tr(\iT(3)|_{\cS^0_k(29)})&=\z\frac23(-3)^\vk(2+B_k)+2a_{1,3,k}+2a_{2,3,k}-4\delta_{k,2},\\
\tr(\iT(3)|_{\cS^0_k(30)})&=\z\frac23(-3)^\vk(1+2B_k)-\delta_{k,2},\\
\tr(\iT(3)|_{\cS^0_k(31)})&=2a_{2,3,k}-4\delta_{k,2},\\
\tr(\iT(3)|_{\cS^0_k(32)})&=0,\\
\tr(\iT(3)|_{\cS^0_k(33)})&=-\z\frac23(-3)^\vk(2+B_k)+\delta_{k,2},\\
\tr(\iT(3)|_{\cS^0_k(34)})&=-\z\frac23(-3)^\vk(1+2B_k)-4a_{1,3,k}+4\delta_{k,2},\\
\tr(\iT(3)|_{\cS^0_k(35)})&=-4a_{2,3,k}+4\delta_{k,2},\\
\tr(\iT(3)|_{\cS^0_k(39)})&=\delta_{k,2},\\
\tr(\iT(3)|_{\cS^0_k(42)})&=-\delta_{k,2}.\end{align*}
\end{ex}

\begin{Lem}\label{Ts}
If $l>1$ is square-free and $(l,t)>1$, then we have $t\notin\T_\square(l)$.
\end{Lem}
\begin{proof}
If $(l,t)>2$, then there exists $p\in\P((l,t))\diagdown\{2\}$ and $t^2-4l$ is not square since $v_p(t^2-4l)=1$. If $(l,t)=2$, then we see $l\equiv2 \bmod 4$ and $t^2-4l$ is not square since $(\frac t2)^2-l\equiv-2$ or ${-1}\bmod4$.
\end{proof}

\begin{Prop}
If $l>1$ is square-free, $(l,N/\!\!/l)=1$ and $(l,N)>2\sqrt l$, then we have
\[\tr(\iT(l)|_{\cS^0_k(N)})=-\z\frac{h(-l)}2(-l)^\vk\Lambda_{0,l}(N)+\delta_{k,2}\mu(N)\VP{p\in\P(l/\!\!/N)}(1+p).\]
\end{Prop}
\begin{proof}
By Theorem \ref{1}, Proposition \ref{Lambda0} and Lemma \ref{Ts}, we see
\[\tr(\iT(l)|_{\cS^0_k(N)})=-h_{0,l}(-l)^\vk\Lambda_{0,l}(N)+\delta_{k,2}\mu(N)\VP{p\in\P(l/\!\!/N)}(1+p).\]
We note $l\ge5$ since $l\ge(l,N)>2\sqrt l$, thus $h_{0,l}=\frac{h(-l)}2$.
\end{proof}

\begin{ex} For each $l,N\le42$ satisfying the conditions of the above Proposition, we write down the following formulas:
\begin{align*}\tr(\iT(5)|_{\cS^0_k(5)})&=(-5)^\vk-\delta_{k,2},\\
\tr(\iT(5)|_{\cS^0_k(10)})&=-(-5)^\vk+\delta_{k,2},\\
\tr(\iT(5)|_{\cS^0_k(15)})&=\delta_{k,2},\\
\tr(\iT(5)|_{\cS^0_k(20)})&=-(-5)^\vk,\\
\tr(\iT(5)|_{\cS^0_k(30)})&=-\delta_{k,2},\\
\tr(\iT(5)|_{\cS^0_k(35)})&=\delta_{k,2},\\
\tr(\iT(5)|_{\cS^0_k(40)})&=(-5)^\vk,\end{align*}
\begin{align*}\tr(\iT(6)|_{\cS^0_k(6)})&=-(-6)^\vk+\delta_{k,2},\\
\tr(\iT(6)|_{\cS^0_k(30)})&=-\delta_{k,2},\\
\tr(\iT(6)|_{\cS^0_k(42)})&=-\delta_{k,2},\end{align*}
\begin{align*}\tr(\iT(7)|_{\cS^0_k(7)})&=(-7)^\vk-\delta_{k,2},\\
\tr(\iT(7)|_{\cS^0_k(14)})&=\delta_{k,2},\\
\tr(\iT(7)|_{\cS^0_k(21)})&=-2(-7)^\vk+\delta_{k,2},\\
\tr(\iT(7)|_{\cS^0_k(28)})&=0,\\
\tr(\iT(7)|_{\cS^0_k(35)})&=-2(-7)^\vk+\delta_{k,2},\\
\tr(\iT(7)|_{\cS^0_k(42)})&=-\delta_{k,2},\end{align*}
\begin{align*}\tr(\iT(10)|_{\cS^0_k(10)})&=-(-10)^\vk+\delta_{k,2},\\
\tr(\iT(10)|_{\cS^0_k(30)})&=2(-10)^\vk-\delta_{k,2},\end{align*}
\begin{align*}\tr(\iT(11)|_{\cS^0_k(11)})&=2(-11)^\vk-\delta_{k,2},\\
\tr(\iT(11)|_{\cS^0_k(22)})&=-(-11)^\vk+\delta_{k,2},\\
\tr(\iT(11)|_{\cS^0_k(33)})&=\delta_{k,2},\end{align*}
\begin{align*}\tr(\iT(13)|_{\cS^0_k(13)})&=(-13)^\vk-\delta_{k,2},\\
\tr(\iT(13)|_{\cS^0_k(26)})&=-(-13)^\vk+\delta_{k,2},\\
\tr(\iT(13)|_{\cS^0_k(39)})&=-2(-13)^\vk+\delta_{k,2},\end{align*}
\begin{align*}\tr(\iT(14)|_{\cS^0_k(14)})&=-2(-14)^\vk+\delta_{k,2},\\
\tr(\iT(14)|_{\cS^0_k(42)})&=-\delta_{k,2},\end{align*}
\begin{align*}\tr(\iT(15)|_{\cS^0_k(15)})&=-2(-15)^\vk+\delta_{k,2},\\
\tr(\iT(15)|_{\cS^0_k(30)})&=-\delta_{k,2},\end{align*}
\begin{align*}\tr(\iT(17)|_{\cS^0_k(17)})&=2(-17)^\vk-\delta_{k,2},,\\
\tr(\iT(17)|_{\cS^0_k(34)})&=-2(-17)^\vk+\delta_{k,2},\end{align*}
\begin{align*}\tr(\iT(19)|_{\cS^0_k(19)})&=2(-19)^\vk-\delta_{k,2},\\
\tr(\iT(19)|_{\cS^0_k(38)})&=-(-19)^\vk+\delta_{k,2},\end{align*}
\begin{align*}\tr(\iT(21)|_{\cS^0_k(21)})&=-2(-21)^\vk+\delta_{k,2},\\
\tr(\iT(21)|_{\cS^0_k(42)})&=2(-21)^\vk-\delta_{k,2},\end{align*}
\begin{align*}\tr(\iT(22)|_{\cS^0_k(11)})&=(-22)^\vk-3\delta_{k,2},\\
\tr(\iT(22)|_{\cS^0_k(22)})&=-(-22)^\vk+\delta_{k,2},\\
\tr(\iT(22)|_{\cS^0_k(33)})&=-2(-22)^\vk+3\delta_{k,2},\end{align*}
\[\tr(\iT(23)|_{\cS^0_k(23)})=3(-23)^\vk-\delta_{k,2},\]
\begin{align*}\tr(\iT(26)|_{\cS^0_k(13)})&=3(-26)^\vk-3\delta_{k,2},\\
\tr(\iT(26)|_{\cS^0_k(26)})&=-3(-26)^\vk+\delta_{k,2},\\
\tr(\iT(26)|_{\cS^0_k(39)})&=-\delta_{k,2},\end{align*}
\[\tr(\iT(29)|_{\cS^0_k(29)})=3(-29)^\vk-\delta_{k,2},\]
\begin{align*}\tr(\iT(30)|_{\cS^0_k(15)})&=-2(-30)^\vk+3\delta_{k,2},\\
\tr(\iT(30)|_{\cS^0_k(30)})&=2(-30)^\vk-\delta_{k,2},\end{align*}
\[\tr(\iT(31)|_{\cS^0_k(31)})=3(-31)^\vk-\delta_{k,2},\]
\[\tr(\iT(33)|_{\cS^0_k(33)})=-2(-33)^\vk+\delta_{k,2}.\]
\begin{align*}\tr(\iT(34)|_{\cS^0_k(17)})&=2(-34)^\vk-3\delta_{k,2},\\
\tr(\iT(34)|_{\cS^0_k(34)})&=-2(-34)^\vk+\delta_{k,2},\end{align*}
\[\tr(\iT(35)|_{\cS^0_k(35)})=-4(-35)^\vk+\delta_{k,2},\]
\[\tr(\iT(37)|_{\cS^0_k(37)})=(-37)^\vk-\delta_{k,2}.\]
\begin{align*}\tr(\iT(38)|_{\cS^0_k(19)})=3(-38)^\vk-3\delta_{k,2},\\
\tr(\iT(38)|_{\cS^0_k(38)})=-3(-38)^\vk+\delta_{k,2},\end{align*}
\begin{align*}\tr(\iT(39)|_{\cS^0_k(13)})&=4(-39)^\vk-4\delta_{k,2},\\
\tr(\iT(39)|_{\cS^0_k(26)})&=4\delta_{k,2},\\
\tr(\iT(39)|_{\cS^0_k(39)})&=-4(-39)^\vk+\delta_{k,2}.\end{align*}
\[\tr(\iT(41)|_{\cS^0_k(41)})=4(-41)^\vk-\delta_{k,2},\]
\begin{align*}\tr(\iT(42)|_{\cS^0_k(14)})&=-2(-42)^\vk+4\delta_{k,2},\\
\tr(\iT(42)|_{\cS^0_k(21)})&=-2(-42)^\vk+3\delta_{k,2},\\
\tr(\iT(42)|_{\cS^0_k(42)})&=2(-42)^\vk-\delta_{k,2},\end{align*}
\end{ex}

\section{proof of theorem \ref{2}}

We put $\N(N^\times)$ the set of all divisors of $N^\times$. For $h,i\in\N(N^\times)$, we define
\[h\tu i=\VP{p\in\P(h/\!\!/i)\cup\P(i/\!\!/h)}p.\]Then $\N(N^\times)$ becomes a group isomorphic to $(\Z/2\Z)^{\sigma(N^\times)}$ and $\ang{\bullet,i}$ becomes a homomorphism $\N(N^\times)\to\{\pm1\}$, i.e.
\[\ang{h,i}\ang{h',i}=\ang{h\tu h',i}.\]
Let $(l,N^\times)=1$. The orthogonal relation of characters induces
\begin{align*}\sigma(N^\times)\tr(\iT(l)|_{\cS_k^0(N;i)})&=\VS{j|N^\times}\Big(\VS{h|N^\times}\ang{h,i\tu j}\Big)\tr(\iT(l)|_{\cS_k^0(N;j)})\\
&=\VS{h|N^\times}\ang{h,i}\VS{j|N^\times}\ang{h,j}\tr(\iT(l)|_{\cS_k^0(N;j)}).\end{align*}
Here, we see
\begin{align*}\ang{h,j}\tr(\iT(l)|_{\cS_k^0(N;j)})&=\VS{f\in\cP_k(N;j)}\ia_l(f)\ang{h,j}\\
&=\VS{f\in\cP_k(N;j)}\ia_l(f)\VP{p\in\P(h)}\frac{\ia_p(f)}{p^\vk}
\\&=h^{-\vk}\tr(\iT(hl)|_{\cS_k^0(N;j)}),\end{align*}
and thus
\[\sigma(N^\times)\tr(\iT(l)|_{\cS_k^0(N;i)})=\VS{h|N^\times}\ang{h,i}h^{-\vk}\tr(\iT(hl)|_{\cS_k^0(N)}).\]

\section{Examples of dimension of $\cS^0_k(N;i)$}

For example, if $p$ is prime and $(p,l)=1$, then we see
\begin{align*}\dim\cS_k^0(p;1)&=\z\frac12\big(\dim\cS_k^0(p)+\frac1{p^\vk}\tr(\iT(p)|_{\cS_k^0(p)})\big),\\
\dim\cS_k^0(p;p)&=\z\frac12\big(\dim\cS_k^0(p)-\frac1{p^\vk}\tr(\iT(p)|_{\cS_k^0(p)})\big).\end{align*}
Put \[d_k(N;i)=\dim\cS^0_k(N;i)-\mu(N)\delta_{k,2}\delta_{i,1}.\]
We calculate $d_k(N;i)$ for each $N\le42$.
\begin{ex}\label{ex}
For prime $N$, we write down the following formulas:
\[\begin{array}{|c|cc|}\hline
k&d_k(2;1)&d_k(2;2)\\\hline
2+24n&1+n&n\\
4+24n&n&n\\
6+24n&n&n\\
8+24n&n&1+n\\
10+24n&1+n&n\\
12+24n&n&n\\\hline
14+24n&1+n&1+n\\
16+24n&n&1+n\\
18+24n&1+n&n\\
20+24n&1+n&1+n\\
22+24n&1+n&1+n\\
24+24n&n&1+n\\\hline
\end{array}\]
\[\begin{array}{|c|cc|cc|cc|}\hline
k&d_k(3;1)&d_k(3;3)&d_k(5;1)&d_k(5;5)&d_k(11;1)&d_k(11;11)\\\hline
2+12n&1+n&n&1+2n&2n&2+5n&5n\\
4+12n&n&n&2n&1+2n&5n&2+5n\\
6+12n&1+n&n&1+2n&2n&3+5n&1+5n\\\hline
8+12n&n&1+n&1+2n&2+2n&2+5n&4+5n\\
10+12n&1+n&1+n&2+2n&1+2n&5+5n&3+5n\\
12+12n&n&1+n&1+2n&2+2n&3+5n&5+5n\\\hline
\end{array}\]
\[\begin{array}{|c|cc|cc|}\hline
k&d_k(17;1)&d_k(17;17)&d_k(23;1)&d_k(23;23)\\\hline
2+12n&2+8n&8n&3+11n&11n\\
4+12n&1+8n&3+8n&1+11n&4+11n\\
6+12n&4+8n&2+8n&6+11n&3+11n\\\hline
8+12n&4+8n&6+8n&5+11n&8+11n\\
10+12n&7+8n&5+8n&10+11n&7+11n\\
12+12n&6+8n&8+8n&8+11n&11+11n\\\hline
\end{array}\]
\[\begin{array}{|c|cc|cc|}\hline
k&d_k(29;1)&d_k(29;29)&d_k(41;1)&d_k(41;41)\\\hline
2+12n&3+14n&14n&4+20n&20n\\
4+12n&2+14n&5+14n&3+20n&7+20n\\
6+12n&7+14n&4+14n&10+20n&6+20n\\\hline
8+12n&7+14n&10+14n&10+20n&14+20n\\
10+12n&12+14n&9+14n&17+20n&13+20n\\
12+12n&11+14n&14+14n&16+20n&20+20n\\\hline
\end{array}\]
\[\begin{array}{|c|cc|cc|cc|}\hline
k&d_k(7;1)&d_k(7;7)&d_k(13;1)&d_k(13;13)&d_k(19;1)&d_k(19;19)\\\hline
2+4n&1+n&n&1+2n&2n&2+3n&3n\\\hline
4+4n&n&1+n&1+2n&2+2n&1+3n&3+3n\\\hline
\end{array}\]
\[\begin{array}{|c|cc|cc|}\hline
k&d_k(31;1)&d_k(31;31)&d_k(37;1)&d_k(37;37)\\\hline
2+4n&3+5n&5n&2+6n&1+6n\\\hline
4+4n&2+5n&5+5n&4+6n&5+6n\\\hline
\end{array}\]
\end{ex}

\begin{ex}For square-free composite $N$, we write down the following formulas:
\[\begin{array}{|c|cccc|}\hline
k&d_k(6;1)&d_k(6;2)&d_k(6;3)&d_k(6;6)\\\hline
2+24n&-1+n&n&n&n\\
4+24n&n&n&n&1+n\\
6+24n&n&n&1+n&n\\
8+24n&1+n&n&n&n\\
10+24n&n&1+n&n&n\\
12+24n&1+n&1+n&n&1+n\\\hline
14+24n&n&n&1+n&n\\
16+24n&1+n&n&1+n&1+n\\
18+24n&n&1+n&1+n&1+n\\
20+24n&1+n&1+n&n&1+n\\
22+24n&1+n&1+n&1+n&n\\
24+24n&2+n&1+n&1+n&1+n\\\hline
\end{array}\]
\[\begin{array}{|c|cc|}\hline
k&d_k(10;1)&d_k(10;2\text{ or }3\text{ or }5)\\\hline
2+12n&-1+n&n\\
4+12n&1+n&n\\
6+12n&n&1+n\\\hline
8+12n&1+n&n\\
10+12n&n&1+n\\
12+12n&2+n&1+n\\\hline
\end{array}\]
\[\begin{array}{|c|cccc|}\hline
k&d_k(14;1)&d_k(14;2)&d_k(14;7)&d_k(14;14)\\\hline
2+8n&-1+n&1+n&n&n\\
4+8n&1+n&n&n&1+n\\\hline
6+8n&n&1+n&1+n&n\\
8+8n&2+n&n&1+n&1+n\\\hline
\end{array}\]
\[\begin{array}{|c|cccc|}\hline
k&d_k(15;1)&d_k(15;3)&d_k(15;5)&d_k(15;15)\\\hline
2+12n&-1+2n&1+2n&2n&2n\\
4+12n&1+2n&2n&2n&1+2n\\
6+12n&2n&2+2n&1+2n&1+2n\\\hline
8+12n&2+2n&2n&1+2n&1+2n\\
10+12n&1+2n&2+2n&2+2n&1+2n\\
12+12n&3+2n&1+2n&2+2n&2+2n\\\hline
\end{array}\]
\[\begin{array}{|c|ccc|}\hline
k&d_k(21;1)&d_k(21;3\text{ or }21)&d_k(21;7)\\\hline
2+4n&-1+n&n&1+n\\\hline
4+4n&2+n&1+n&n\\\hline
\end{array}\]
\[\begin{array}{|c|ccc|}\hline
k&d_k(22;1)&d_k(22;2\text{ or }22)&d_k(22;11)\\\hline
2+24n&-1+5n&5n&5n\\
4+24n&1+5n&1+5n&5n\\
6+24n&1+5n&1+5n&2+5n\\
8+24n&2+5n&1+5n&1+5n\\
10+24n&1+5n&2+5n&2+5n\\
12+24n&3+5n&3+5n&2+5n\\\hline
14+24n&2+5n&2+5n&3+5n\\
16+24n&4+5n&3+5n&3+5n\\
18+24n&3+5n&4+5n&4+5n\\
20+24n&4+5n&4+5n&3+5n\\
22+24n&4+5n&4+5n&5+5n\\
24+24n&6+5n&5+5n&5+5n\\\hline
\end{array}\]
\[\begin{array}{|c|ccc|}\hline
k&d_k(26;1)&d_k(26;2\text{ or }13)&d_k(26;26)\\\hline
2+4n&-1+n&1+n&n\\\hline
4+4n&2+n&n&1+n\\\hline
\end{array}\]
\[\begin{array}{|c|cccc|}\hline
k&d_k(33;1)&d_k(33;3)&d_k(33;11)&d_k(33;33)\\\hline
2+12n&-1+5n&1+5n&5n&5n\\
4+12n&2+5n&1+5n&1+5n&2+5n\\
6+12n&1+5n&3+5n&2+5n&2+5n\\\hline
8+12n&4+5n&2+5n&3+5n&3+5n\\
10+12n&3+5n&4+5n&4+5n&3+5n\\
12+12n&6+5n&4+5n&5+5n&5+5n\\\hline
\end{array}\]
\[\begin{array}{|c|ccc|}\hline
k&d_k(34;1)&d_k(34;2\text{ or }34)&d_k(34;17)\\\hline
2+12n&-1+4n&4n&1+4n\\
4+12n&2+4n&1+4n&4n\\
6+12n&1+4n&2+4n&3+4n\\\hline
8+12n&3+4n&2+4n&1+4n\\
10+12n&2+4n&3+4n&4+4n\\
12+12n&5+4n&4+4n&3+4n\\\hline
\end{array}\]
\[\begin{array}{|c|cccc|}\hline
k&d_k(35;1)&d_k(35;5)&d_k(35;7)&d_k(35;35)\\\hline
2+4n&-1+2n&1+2n&2+2n&2n\\\hline
4+4n&3+2n&1+2n&2n&2+2n\\\hline
\end{array}\]
\[\begin{array}{|c|cccc|}\hline
k&d_k(38;1)&d_k(38;2)&d_k(38;19)&d_k(38;38)\\\hline
2+8n&-1+3n&1+3n&1+3n&3n\\
4+8n&2+3n&1+3n&3n&2+3n\\\hline
6+8n&1+3n&2+3n&3+3n&1+3n\\
8+8n&4+3n&2+3n&2+3n&3+3n\\\hline
\end{array}\]
\[\begin{array}{|c|cccc|}\hline
k&d_k(39;1)&d_k(39;3)&d_k(39;13)&d_k(39;39)\\\hline
2+4n&-1+2n&1+2n&2+2n&2n\\\hline
4+4n&3+2n&1+2n&2n&2+2n\\\hline
\end{array}\]
\[\begin{array}{|c|ccc|}\hline
k&d_k(30;1\text{ or }10)&d_k(30;3\text{ or }6\text{ or }15\text{ or }30)&d_k(30;2\text{ or }5)\\\hline
2+12n&1+n&n&n\\
4+12n&n&n&1+n\\
6+12n&1+n&n&n\\\hline
8+12n&n&1+n&1+n\\
10+12n&1+n&1+n&n\\
12+12n&n&1+n&1+n\\\hline
\end{array}\]
\[\begin{array}{|c|ccc|}\hline
k&d_k(42;1\text{ or }21)&d_k(42;2\text{ or }6\text{ or }14\text{ or }42)&d_k(42;3\text{ or }7)\\\hline
2+8n&1+n&n&n\\
4+8n&n&n&1+n\\\hline
6+8n&1+n&1+n&n\\
8+8n&n&1+n&1+n\\\hline
\end{array}\]
\end{ex}
\smallskip
\begin{ex}
For not square-free $N$, we write down the following formulas:
\[\begin{array}{|c|cc|cc|}\hline
k&d_k(12;1)&d_k(12;3)&d_k(20;1)&d_k(20;5)\\\hline
2+12n&n&n&2n&1+2n\\
4+12n&1+n&n&1+2n&2n\\
6+12n&n&n&2n&1+2n\\\hline
8+12n&1+n&1+n&2+2n&1+2n\\
10+12n&n&1+n&1+2n&2+2n\\
12+12n&1+n&1+n&2+2n&1+2n\\\hline
\end{array}\]
\[\begin{array}{|c|cc|}\hline
k&d_k(18;1)&d_k(18;2)\\\hline
2+24n&5n&5n\\
4+24n&1+5n&5n\\
6+24n&1+5n&2+5n\\
8+24n&1+5n&1+5n\\
10+24n&2+5n&2+5n\\
12+24n&3+5n&2+5n\\\hline
14+24n&2+5n&3+5n\\
16+24n&3+5n&3+5n\\
18+24n&4+5n&4+5n\\
20+24n&4+5n&3+5n\\
22+24n&4+5n&5+5n\\
24+24n&5+5n&5+5n\\\hline
\end{array}\]
\[\begin{array}{|c|cc|c|cc|}\hline
k&d_k(24;1)&d_k(24;3)&d_k(28;1\text{ or }7)&d_k(40;1)&d_k(40;5)\\\hline
2+4n&n&1+n&n&1+2n&2n\\\hline
4+4n&1+n&n&1+n&1+2n&2+2n\\\hline
\end{array}\]
\end{ex}

\section{Some primitive forms for $N=14$}

\subsection{Modular forms}

We recall some facts on modular forms, for the next subsection. For a congruence subgroup $\Gamma$ of ${\rm SL}_2(\Z)$, we denote by $\cM_k(\Gamma)$ the space of all modular forms of wight $k$ with respect to $\Gamma$. We put $\cM_k(N)=\cM_k(\Gamma_0(N))$. Moreover, we define
\[\Gamma_1(N)=\{(\BM a&b\\c&d\EM)\in\Gamma_0(N) \,|\, a\equiv1\bmod N\},\]
and for each Dirichlet character $\chi:(\Z/N\Z)^\times\to\C^\times$
\[\cM_k(N,\chi)=\big\{f\in\cM_k(\Gamma_1(N)) \,\big|\, (cz+d)^{-k}f(\z\frac{az+b}{cz+d})=\chi(d)f\text{ for all }(\BM a&b\\c&d\EM)\in\Gamma_0(N)\big\}.\]
Put
\begin{align*}\iC_2&=1+24\VS{n=1}^\infty\Big(\VS{d|n}\bm1_2(d)d\Big)q^n\in\cM_2(2),\\
\iF_7&=1+2\VS{n=1}^\infty\Big(\VS{d|n}\rho_7(d)\Big)q^n\in\cM_1(7,\rho_7),\end{align*}
where $\bm1_N:(\Z/N\Z)^\times\to\{1\}$ and $\bm1_7\neq\rho_7:(\Z/7\Z)^\times\to\{1,-1\}$. See \cite[\S4]{DS} or \cite[\S5.3]{Ste} for details.

We regard $\cM_k(\Gamma_1(N))\subset\C[[q]]$ via the Fourier expansion.
\begin{Lem}\label{deg}We see
\[\cM_2(14)\cap\C[[q]]q^4=\{0\}.\]\end{Lem}
\begin{proof}Put
\begin{align*}\alpha&=\z\frac12(\iF_7-\iF_7^{(2)})\in\cM_1(14,\rho_7)\cap\C[[q]]q,\\
\gamma&=\z\frac18(\iF_7^2-2\iF_7\iF_7^{(2)}+\iC_2^{(7)})\in\cM_2(14)\cap\C[[q]]q^3,\end{align*}
where $f^{(h)}(q)=f(q^h)$. Since $\dim\cM_2(14)=4$ (cf. \cite[Theorem 3.5.1]{DS}), we see
\[\cM_2(14)=\C\iF_7^2\oplus\C\iF_7\alpha\oplus\C\alpha^2\oplus\C\gamma,\]
and we easily get the assertion.
\end{proof}
Note that a weaker result
\[\cM_2(14)\cap\C[[q]]q^5=\{0\}\]
can be obtained directly from a result of Sturm \cite{Stu} and $[{\rm SL}_2(\Z):\Gamma_0(14)]=24$.

\smallskip

\subsection{Primitive forms}

We represent all primitive forms of weight 2,4, in terms of $\iC_2$ and $\iF_7$. First, put
\[\Delta=\z\frac12\big(\iF_7-\iF_7^{(2)}\big)\big(2\iF_7^{(2)}-\iF_7\big).\]

\begin{ex}We see
\[\cP_2(14;2)=\{\Delta\}.\]
\end{ex}
\begin{proof}First, we easily see
\[\Delta\in\cM_2(14)\cap(q-q^2-2q^3+\C[[q]]q^4).\]
Note $\cS_2(14)=\cS_2(14;2)$ and $\dim\cS_2(14)=1$ by Example \ref{ex}. Let $\cP_2(14;2)=\{f\}$, then we see $\ia_1(f)=1$ and
\begin{align*}\ia_2(f)&=-2^0=-1,\\
\ia_3(f)&=\tr(\iT(3)|_{\cS_2(14)})=-2.\end{align*}
Thus, we see
\[f-\Delta\in\cM_2(14)\cap\C[[q]]q^4=\{0\}\]
i.e. $f=\Delta$, and get the assertion.\end{proof}
\smallskip
We remark that $\Delta$ may be represented as a multiplicative $\eta$-product (cf. \cite{DKM}).
\begin{Lem}\label{gen}For $k\ge0$, we get
\[\cS_{k+2}(14)=\Delta\cM_k(14).\]
\end{Lem}
\begin{proof}
The asseriton follows from $\cS_{k+2}(14)\supset\Delta\cM_k(14)$ and
\[\dim\cS_{k+2}(14)=\dim\cM_k(14)=\dim(\Delta\cM_k(14)).\]
\end{proof}
\begin{ex}We see
\begin{align*}\cP_4(14;1)&=\big\{\z\frac18\Delta(\iC_2+7\iC_2^{(7)})\big\},\\
\cP_4(14;14)&=\big\{\z\frac12\Delta(3\iF_7^2-7\iF_7\iF_7^{(2)}+6\iF_7^{(2)2})\big\}.\end{align*}
\end{ex}
\begin{proof}
At first, we see
\[\cS_4(14)\cap\C[[q]]q^5=\Delta\big(\cM_2(14)\cap\C[[q]]q^4\big)=\{0\}\]
by Lemma \ref{gen} and \ref{deg}.
Note that $\dim\cS_4(14;1)=1$ by Example \ref{ex}. Let $\cP_4(14;1)=\{f\}$, then we see $\ia_1(f)=1$, $\ia_2(f)=2$ and $\ia_4(f)=\ia_2(f)^2=4$. In addition, we see
\[\ia_3(f)=\z\frac14(6-10-10+6)=-2\]
by Theorem \ref{2} and
\begin{align*}\tr(\iT(3)|_{\cS_4(14)})&=6,\\
\z\frac12\tr(\iT(6)|_{\cS_4(14)})&=-10,\\
\z\frac17\tr(\iT(21)|_{\cS_4(14)})&=-10,\\
\z\frac1{14}\tr(\iT(42)|_{\cS_4(14)})&=6.\end{align*}
Thus, we get
\[f-\z\frac18\Delta(\iC_2+7\iC_2^{(7)})\in\cS_4(14)\cap\C[[q]]q^5=\{0\}.\]
We get the second assertion in a similar way.
\end{proof}
For many other examples for $N=1,2,3,4,6,8,9$, see \cite{S2}.


\begin{thebibliography}{99}
\bibitem{AL} A.O.L.Atkin \& J.Lehner, Hecke operators on $\Gamma_0(m)$, Mathematische Annalen {\bf 185}(1970), 134-160.
\bibitem{DKM} D.Dummit, H.Kisilevsky \& J.McKay, Multiplicative $\eta$-products, Comtemp. Math. {\bf 45}(1985), 89-98.
\bibitem{DS} F.Diamond \& J.Shurman, A First Course in Modular Forms, GTM {\bf 228}, Springer, 2005.
\bibitem{E} M.Eichler, Eine Verallgenmeinerung der Abelschen Integrale, Math.Z {\bf 67}(1957), 267-298.
\bibitem{H} E.Hecke, Die Klassenzahl imagin\"ar-quadratischer K\"orper in der theorie der elliptischen modularfunktionen, Monat. f\"ur Math. und Phys {\bf 48}(1939), 75-83 (=Math. Werke, 773-781).
\bibitem{K} K.Kubo, A generalization of a result of Hecke and dimensions of Hecke submodules, unpublished, around 1998.
\bibitem{Ma} G.Martin, Dimensions of the spaces of cusp forms and newforms on $\Gamma_0(N)$ and $\Gamma_1(N)$, J.Number Theory {\bf 112}(2005), 298-331.
\bibitem{Mi} T.Miyake, Modular Forms, Springer-Verlag, 1989.
\bibitem{Stu} J.Sturm, On the congruence of modular forms, Lecture Notes in Math, Vol {\bf 1240}.
\bibitem{S2} T.Suda, An explicit representation of primitive forms, preparing.
\end{thebibliography}
\end{document}